\documentclass[10pt]{article}

\font\smallit=cmti10

\usepackage{amssymb,latexsym,amsmath,epsfig,amsthm}
\usepackage{amsfonts}
\usepackage[justification=centering]{caption}
\usepackage{float}
\usepackage{afterpage}
\usepackage{enumerate}
\usepackage{hyperref}
\usepackage{xcolor}
\makeatletter

\renewcommand\section{\@startsection {section}{1}{\z@}
{-30pt \@plus -1ex \@minus -.2ex}
{2.3ex \@plus.2ex}
{\normalfont\normalsize\bfseries\boldmath}}

\renewcommand\subsection{\@startsection{subsection}{2}{\z@}
{-3.25ex\@plus -1ex \@minus -.2ex}
{1.5ex \@plus .2ex}
{\normalfont\normalsize\bfseries\boldmath}}

\renewcommand{\@seccntformat}[1]{\csname the#1\endcsname. }

\makeatother

\newtheorem{theorem}{Theorem}
\newtheorem{lemma}{Lemma}

\newtheorem{corollary}{Corollary}

\allowdisplaybreaks

\begin{document}

\begin{center}
{\Large \bf On $n^{\rm th}$ order Euler polynomials}
\vskip 5pt
{\Large \bf of degree $n$ that are Eisenstein}
\vskip 20pt
{\bf Michael Filaseta}\\
{\smallit Dept.~Mathematics, 
University of South Carolina, 
Columbia, SC 29208, USA}\\
{\tt filaseta@math.sc.edu}\\ 
\vskip 20pt
{\bf Thomas Luckner}\\
{\smallit Dept.~Mathematics, 
University of South Carolina, 
Columbia, SC 29208, USA}\\
{\tt luckner@email.sc.edu}\\ 
\end{center}

\vskip 5pt
\centerline{\phantom{\smallit Received: , Revised: , Accepted: , Published: }} 

\vskip 5pt

\centerline{\bf Abstract}
\vskip 5pt\noindent
For $m$ an even positive integer and $p$ a prime, we show that the generalized Euler polynomial $E_{mp}^{(mp)}(x)$ is in Eisenstein form with respect to $p$ if and only if 
$p$ does not divide $m (2^m-1)B_m$.  As a consequence, 
we deduce that at least $1/3$ of the generalized Euler polynomials $E_n^{(n)}(x)$ are in Eisenstein form with respect to a prime $p$ dividing $n$ and, hence, irreducible over $\mathbb Q$. 

\pagestyle{myheadings} 
\thispagestyle{empty} 
\baselineskip=12.875pt 
\vskip 30pt
\section{Introduction}

For $m$ a positive integer, the $m$th order Bernoulli polynomial of degree $n$, denoted $B_{n}^{(m)}(x)$, and  the $m$th order Euler polynomial of degree $n$, denoted $E_{n}^{(m)}(x)$, are defined by 
\begin{equation*}
\left(\dfrac{t}{e^t-1}\right)^m e^{tx}=\sum_{n=0}^{\infty}B_n^{(m)}(x)\dfrac{t^n}{n!}
\end{equation*}
and
\begin{equation}\label{eulerdef}
\left(\dfrac{2}{e^t+1}\right)^m e^{tx}=\sum_{n=0}^{\infty}E_n^{(m)}(x)\dfrac{t^n}{n!},
\end{equation}
respectively. We will also want to make use of the Bernoulli number $B_n=B_n^{(1)}(0)$.
Let $\nu_{p}$ denote the usual $p$-adic valuation so that, in particular, for non-zero integers $a$ and $b$, we have $\nu_{p}(a) = k \in \mathbb Z^{+} \cup \{ 0 \}$ means that $p^{k} \mid a$ and $p^{k+1} \nmid a$ and $\nu_{p}(a/b) = \nu_{p}(a) - \nu_{p}(b)$.   
A polynomial $f(x) \in \mathbb Q[x]$ is said to be \textit{Eisenstein} if there is an integer $a$ and a prime $p$ for which the well-known Eisenstein criterion applies to $f(x+a) = \sum_{j=0}^{n} a_{j} x^{j}$ so that $\nu_{p}(a_{n}) = 0$, $\nu_{p}(a_{j}) \ge 1$ for $j \in \{ 0, 1, \ldots, n-1 \}$, and $\nu_{p}(a_{0}) = 1$.  
Eisenstein's criterion implies that Eisenstein polynomials are irreducible over $\mathbb Q$ if $n \ge 1$.  
In the special case that $a = 0$, we say that \textit{$f(x)$ is in Eisenstein form} or, if $p$ is fixed, that \textit{$f(x)$ is in Eisenstein form with respect to $p$}.  
A.~Adelberg and the first author \cite{adelfil} showed, somewhat surprisingly, that Eisenstein's criterion applies to many of the $n$th order Bernoulli polynomials of degree $n$.  
More precisely, they showed that
\begin{equation}\label{adelfilresult}
\liminf_{t \rightarrow \infty} \dfrac{|\{ n \le t: B_{n}^{(n)}(x) \text{ is in Eisenstein form} \}|}{t} > \dfrac{1}{5}.
\end{equation}
Experimentally, the authors noticed that the Euler polynomials $E_{n}^{(n)}(x)$ often also appear to be in Eisenstein form.  The polynomials $E_{n}^{(n)}(x)$ have been investigated less in the literature, so we are not as readily able to apply known results to derive such a result.  
However, as we will see, we are still able to establish the following results.

\begin{theorem}\label{eisensteinpolynsthm}
Let $m$ be an even positive integer and $p$ be an odd prime. 
Then $E_{mp}^{(mp)}(x)$ is in Eisenstein form with respect to $p$ if and only if 
$p$ does not divide $m (2^m-1)B_m$.
\end{theorem}

To clarify, since $m$ is even, the number $m (2^m-1)B_m$ appearing in Theorem~\ref{eisensteinpolynsthm} is an integer (cf.~\cite[Lemma~1]{adelfil}).

\begin{theorem}\label{positivepropthm}
Asymptotically, more than one-third of the polynomials $E_{n}^{(n)}(x)$ are irreducible (and in fact Eisenstein).  More precisely, 
\[
\liminf_{t \rightarrow \infty} \dfrac{|\{ n \le t: E_{n}^{(n)}(x) \text{ is in Eisenstein form} \}|}{t} \ge \dfrac{\log 2}{2} = 0.34657\ldots.
\]
\end{theorem}

For each odd $n \ge 1$, the polynomial $E_n^{(n)}(x)$ has a linear factor.  To see this, it suffices to show, for such $n$, that $E_n^{(n)}(n/2) = 0$. By \eqref{eulerdef}, we obtain
\[
\sum_{n=0}^{\infty}E_n^{(m)}(m/2)\dfrac{t^n}{n!}
= \left(\dfrac{2}{e^t+1}\right)^m e^{tm/2}
= \left(\dfrac{2}{e^{t/2}+e^{-t/2}}\right)^m, 
\]
which is an even function of $t$. Thus, in fact, we have $E_n^{(m)}(m/2) = 0$ for all positive integers $n$ and $m$, with $n$ odd.

Since $E_{1}^{(1)}(x) = x - 1/2$ and $E_{n}^{(n)}(x)$ has a linear factor for $n$ odd, the polynomials $E_{n}^{(n)}(x)$ are not Eisenstein for $n$ odd.  
Thus, the value of $(\log 2)/2$ cannot be replaced by a constant $> 1/2$ in Theorem~\ref{positivepropthm}.  
Based on the Siegel heuristic (cf.~\cite[Chapter~6, VI]{ribenboim}) that the residue classes of the numerators of $B_{2n}$ are randomly distributed 
and based on lower bounds for the order of $2$ modulo most primes \cite{murty}, we deduce from Theorem~\ref{eisensteinpolynsthm} that, when $n$ is even, typically $E_{n}^{(n)}(x)$ will be in Eisenstein form with respect to the largest prime divisor of $n$.  
More precisely, assuming the Siegel heuristic, Theorem~\ref{eisensteinpolynsthm} implies that the exact density of $n$ for which the largest prime
divisor $p$ of $n$ is odd and $p$ does not divide $m (2^m-1)B_m$, where $m = n/p$, is $1/2$.
This heuristic argument is supported by the data in Table~\ref{onlytable} obtained by applying Theorem~\ref{eisensteinpolynsthm} together with Kummer's congruence \cite{kummer0} to speed up computations.

\begin{table}[ht]
\caption{Proportion of $E_{n}^{(n)}(x)$ in Eisenstein form with respect to $n$'s largest prime divisor $p$}
\begin{center}
\renewcommand{\tabcolsep}{10pt}
\renewcommand{\arraystretch}{1.2}
\begin{tabular}{|c|c|c|}
\hline
upper bound on $n$ & \# in Eisenstein form w.r.t.~$p$ & percentage \\
\hline
$10^{3}$ & 435 & 0.435 \\
$10^{4}$ & 4642 & 0.4642 \\
$10^{5}$ & 48026 & 0.48026 \\
$10^{6}$ & 489165 & 0.489165 \\
$10^{7}$ & 4939774 & 0.4939774 \\
$10^{8}$ & 49662803 & 0.49662803 \\
\hline
\end{tabular}
\end{center}
\label{onlytable}
\end{table}

The main part of the paper is devoted to the proof of Theorem~\ref{eisensteinpolynsthm}.  The proof of Theorem~\ref{positivepropthm} is based on applying Theorem~\ref{eisensteinpolynsthm} in the case that $p > m$. As a positive proportion of $n$ have a prime factor greater than $\sqrt{n}$, it is reasonable to expect that one can deduce a positive proportion of $n$ can be shown to be Eisenstein in this manner.  The ideas for the proof of Theorem~\ref{positivepropthm} are closely related to arguments in \cite{adelfil}.  However, we are able to get a better density bound by modifying the arguments slightly.  The better bound applies to the case of the Bernoulli polynomials $B_n^{(n)}(x)$ dealt with in \cite{adelfil} and as a consequence the right-hand side of \eqref{adelfilresult} can be replaced with $(\log 2)/2$.  As the argument for this sharpening of \eqref{adelfilresult} is essentially identical to our proof of Theorem~\ref{positivepropthm}, we do not elaborate on improving \eqref{adelfilresult} further.

Experiments suggested that for most even numbers $n$ and primes $p$ for which $p$ divides $n$ but $p^2$ does not, the polynomial $E_{n}^{(n)}(x)$ is Eisenstein with respect to $p$. But we did encounter examples where that was not the case, such as $n = 8\cdot 17$ where $E_{n}^{(n)}(x)$ is not Eisenstein with respect to $17$. As was pointed out to us by Pieter Moree, there are connections to these observations with recent papers on Genocchi numbers $G_n = 2(1-2^n)B_n$ \cite{moree1,moree2}.  To be precise, an odd prime $p$ is \textit{$G$-regular} if $p$ does not divide each of $G_{2}, G_{4}, \ldots, G_{p-3}$ and otherwise is \textit{$G$-irregular}. As a consequence of Theorem~\ref{eisensteinpolynsthm} and \cite[Lemma~4, Proposition~1]{moree2}, if $n = mp$ where $m$ is even and $p \nmid m$, then $E_{mp}^{(mp)}(x)$ is in Eisenstein form with respect to $p$ if $p$ is $G$-regular and $p$ is not a Wieferich prime (that is, a prime $p$ for which $p^2$ divides $2^{p-1}-1$). The Wieferich primes up to $6.7 \cdot 10^{15}$ are simply $1093$ and $3511$ both of which are $G$-irregular.  The Wieferich primes play a different role in that these are the only primes $p$ where if $p-1$ divides $m$ and $p \nmid m$, then $E_{mp}^{(mp)}(x)$ is not in Eisenstein form with respect to $p$. The prime $17$ is $G$-irregular accounting for the example $n = 8 \cdot 17$ above. Despite the example, we note that often $E_{mp}^{(mp)}(x)$ will be in Eisenstein form when $p$ is $G$-irregular (for example, $E_{mp}^{(mp)}(x)$ is in Eisenstein form for $p = 17$ and $m \in \{ 2, 4, 6 \}$).  The following is a direct consequence of Theorem~\ref{eisensteinpolynsthm} above and Theorem~1.10 in \cite{moree1}.

\begin{theorem}\label{newthm}
Let $\varepsilon > 0$ be arbitrary and fixed. The number of primes $2 < p \le x$
for which $E_{mp}^{(mp)}(x)$ is not in Eisenstein form with respect to $p$ for some even
$m$ coprime to $p$ is at least
\[
\bigg( 1 - \dfrac{3A}{2} - \varepsilon \bigg) \dfrac{x}{\log x},
\]
where $A$ is the Artin constant
\[
A =
\prod_{p \text{ prime}}
\bigg( 1 - \dfrac{1}{p(p-1)}\bigg)
= 0.3739558136192022880547280543464\ldots.
\]
\end{theorem}

Note that $1-3A/2 = 0.4390662795\ldots$.  Thus, the set of $G$-irregular primes has density at least $0.439$.  According to \cite{moree1}, the list of $G$-irregular primes begins with
\[
17, 31, 37, 41, 43, 59, 67, 73, 89, 97, 101, \ldots,
\]
and Siegel's heuristics suggest that the correct proportion of primes which are $G$-irregular is
\[
1 - \dfrac{3A}{2\sqrt{e}}
= 0.6597765\ldots.
\]   

\section{Background}

Recall that the $m$th order Euler polynomial of degree $n$, denoted $E_n^{(m)}(x)$, is defined by \eqref{eulerdef}.  
The left-hand side of the equation is to be interpreted as the formal product of Maclaurin series in $t$.  
We will focus on the generalized Euler polynomials $E_n^{(n)}(x)$, or in other words, the case $m=n$. 

We obtain information on the Maclaurin series of $2/(e^t+1)$ as follows.  Observe that
\[
\dfrac{2}{e^t + 1} - 1
= \dfrac{1-e^t}{1+e^t}
= \dfrac{e^{-t/2} - e^{t/2}}{e^{-t/2} + e^{t/2}}
= - \tanh(t/2).
\]
The Maclaurin series for the hyperbolic tangent is well understood, and in particular we deduce (cf.~\cite{abramsteg})
\[
\dfrac{2}{e^t + 1} - 1 = \sum_{n=0}^{\infty} \dfrac{-2 (2^{2n}-1) B_{2n} t^{2n-1}}{(2n)!},
\]
where $B_{2n}$ denotes the $2n^{\rm th}$ Bernoulli number.  Therefore, we see that
\[
\dfrac{2}{e^t + 1} 
= \sum_{n=0}^{\infty} \dfrac{\ell_n}{n! \,2^{n}} \,t^n,
\]
where
\[
\ell_{0} = 1, \quad
\ell_{2n} = 0 \ \text{ for } \ n \ge 1, \quad
\text{and} \quad
\ell_{2n-1} = \dfrac{-2^{2n-1} (2^{2n}-1) B_{2n}}{n} \ \text{ for } \ n \ge 1.
\]
It is known that $\ell_{n} \in \mathbb Z$ for all $n \ge 0$ (cf.~the discussion of $C_v$ in \cite[pp.~27--28]{norlund}).

Since $e^{tx} = \sum_{m=0}^{\infty} x^m t^m/m!$, we see from \eqref{eulerdef} that $E_{n}^{(n)}(x)$ is $n!$ times the coefficient of $t^{n}$ in the expression
\begin{equation*}
\left(\sum_{j=0}^\infty \dfrac{\ell_j}{j!\,2^{j}}t^j\right)^{n}\left(\sum_{m=0}^{\infty}\dfrac{x^m t^m}{m!}\right).
\end{equation*}
For an integer $k \in [0,n]$, the Multinomial Theorem implies the coefficient of $t^{n-k}$ in the power in the expression above is 
\[
\sum_{\substack{e_{1} + 2 e_{2} + \cdots + n e_{n} = n-k \\ e_{0} = n - e_{1} - e_{2} - \cdots - e_{n} \\ e_{1} \ge 0, \ldots, e_{n} \ge 0}} \dfrac{n!}{e_{0}! e_{1}! e_{2}! \cdots e_{n}!} \prod_{j = 1}^{n} \bigg(  \dfrac{\ell_{j}}{j! \,2^{j}}  \bigg)^{e_{j}},
\]
where in the product if $\ell_j = e_j = 0$, then $(\ell_j/(j!\,2^{j}))^{e_j}$ is to be interpreted as $1$.  
We deduce that 
\[
E_{n}^{(n)}(x) = \sum_{k= 0}^{n} E_{n,k} x^{k},
\]
where
\begin{equation}\label{eulercoeffformula}
E_{n,k} = \dfrac{n!}{k!}
 \sum_{\substack{e_{1} + 2 e_{2} + \cdots + n e_{n} = n-k \\ e_{0} = n - e_{1} - e_{2} - \cdots - e_{n} \\ e_{1} \ge 0, \ldots, e_{n} \ge 0}} \dfrac{n!}{e_{0}! e_{1}! e_{2}! \cdots e_{n}!} \prod_{j = 1}^{n} \bigg(  \dfrac{\ell_{j}}{j! \,2^{j}}  \bigg)^{e_{j}}.
\end{equation}
Note that $E_{n}^{(n)}(x)$ is a monic polynomial with rational coefficients.  Given that 
$e_{1} + 2 e_{2} + \cdots + n e_{n} = n-k$ and $e_{0} = n - e_{1} - e_{2} - \cdots - e_{n}$ in the sums above, the expressions
\[
\dfrac{n!}{k!} \prod_{j = 1}^{n} \bigg(  \dfrac{1}{j!} \bigg)^{e_{j}} 
\qquad \text{and} \qquad
\dfrac{n!}{e_{0}! e_{1}! e_{2}! \cdots e_{n}!}
\]
can be viewed as multinomial coefficients and hence integers.  Thus, the coefficients of $E_{n}^{(n)}(x)$ times a power of $2$ will lie in $\mathbb Z$.  We deduce that, for some $N = N(n) \in \mathbb Z^{+}$, we have 
\begin{equation}\label{eulercoeffdef}
E_{n,n} = 1 \quad \text{ and } \quad 2^{N} E_{n}^{(n)}(x) \in \mathbb Z[x].
\end{equation}

\section{Preliminaries for {$n=mp$}}

For the rest of this paper, assume $n=mp$ where $p$ is an odd prime and $m$ is a positive even integer.  Our next goal is to establish Theorem~\ref{eisensteinpolynsthm}.

For $0 \le k \le mp$, we obtain from \eqref{eulercoeffformula} that
\begin{equation}\label{eulercoeffformulam}
E_{mp,k} = \dfrac{(mp)!}{k!}
 \sum_{\substack{e_{1} + 2 e_{2} + \cdots + mp e_{mp} = mp-k \\ e_{0} = mp - e_{1} - e_{2} - \cdots - e_{mp} \\ e_{1} \ge 0, \ldots, e_{mp} \ge 0}} \dfrac{(mp)!}{e_{0}! e_{1}! e_{2}! \cdots e_{mp}!} \prod_{j = 1}^{mp} \bigg(  \dfrac{\ell_{j}}{j! \,2^{j}}  \bigg)^{e_{j}}.
 \end{equation}
Hence, \eqref{eulercoeffdef} holds with $n = mp$, and 
for some $N=N(m,p)\in \mathbb{Z}^+$, we have
 \begin{equation}\label{eulercoeffdefm}
E_{mp,mp}=1 \quad \text{ and } \quad 2^N E_{mp}^{(mp)}(x) \in \mathbb{Z}[x].
 \end{equation}

Observe that $E_{mp}^{(mp)}(x)$ is in Eisenstein form with respect to the odd prime $p$ if and only if each of the following holds:
\begin{enumerate}[(a)]
    \item $p\nmid 2^N E_{mp, mp}$,
    \item $p\mid 2^N E_{mp, k}$ for all $0\le k\le mp-1$, and
    \item $p^2\nmid 2^N E_{mp, 0}$.
\end{enumerate}

\begin{proof}[Proof that part (a) always holds]
We obtain from \eqref{eulercoeffdefm} that $2^N E_{mp,mp}=2^N$. Since $p$ is an odd prime, $p\nmid 2^N E_{mp, mp}$.
\end{proof}

\begin{proof}[Proof that part (b) always holds]
Recall $\ell_j \in \mathbb Z$ for all $j$ and $\ell_{2n}=0$ for positive integer $n$. From \eqref{eulercoeffformulam}, it suffices to show $p$ divides at least one of the multinomial coefficients
\[
a_k\left(e_1, \ldots, e_{mp}\right) = \dfrac{(mp)!}{k!} \prod_{j = 1}^{mp} \bigg(  \dfrac{1}{j!} \bigg)^{e_{j}} 
\quad \text{and} \quad
b_k\left(e_1, \ldots, e_{mp}\right)=\dfrac{(mp)!}{e_{0}! e_{1}! e_{2}! \cdots e_{mp}!}
\]
for each integer $k \in [0,mp)$ and each possible set $\{e_1, \ldots, e_{mp}\}$ of non-negative integers with $e_{1} + 2 e_{2} + \cdots + mp e_{mp} = mp-k$.
To do so, we will make use of the following lemma due to E.~E.~Kummer \cite{kummer}.

\begin{lemma}[Kummer \cite{kummer}]\label{kummerbinomial}
Let $n$ and $u$ be integers with $n \ge u \ge 0$, and let $p$ be a prime.  
If $v$ is the number of carries when adding $u$ and $n-u$ in base $p$, then
\[
\nu_{p}\bigg( \binom{n}{u} \bigg) = v.
\]
\end{lemma}

\begin{corollary}\label{kummerbincor}
Let $n$ be a positive integer, and let $u_1, u_2, \ldots, u_r$ be non-negative integers such that $n=u_1+u_2+\cdots +u_r$. Then 
\[
  \nu_p\left(\dfrac{n!}{u_1!\cdots u_r!}\right) = v
\]
where $v$ is the number of carries when performing the additions $u_1 + u_2 + \ldots + u_r$ in base $p$ from left to right.
\end{corollary}

Corollary~\ref{kummerbincor} is an immediate consequence of Lemma~\ref{kummerbinomial} and the identity
\[
\dfrac{n!}{u_1!\cdots u_r!} = \binom{u_1+u_2}{u_2}\binom{u_1+u_2+u_3}{u_3} \cdots \binom{u_1+u_2+ \ldots + u_r}{u_r}.
\]
Observe that the corollary implies that the number of carries in Corollary~\ref{kummerbincor} is independent of the order in which we add the numbers $u_1, u_2, \ldots, u_r$ in base $p$.
Our main interest in Corollary~\ref{kummerbincor} is the case where $v = 0$, which occurs precisely when there are no carries when adding the numbers $u_1, u_2, \ldots, u_r$ in base $p$.  To put this another way, let $d_j^{(i)} \in \{ 0, 1, \ldots, p-1\}$ (digits in base $p$) for all $i$ and $j$ with $0 \le i \le r$ and $0 \le j \le n_p = \lfloor \log n/\log p \rfloor$.  Suppose
\[
n = \sum_{j=0}^{n_p} d_j^{(0)} p^j
\quad \text{ and } \quad
u_i = \sum_{j=0}^{n_p} d_j^{(i)} p^j \text{ for } i \in \{ 1, \ldots, r \}.
\]
Then $v = 0$ precisely when
\begin{equation}\label{digitsomeforpartb}
d_j^{(0)} = d_j^{(1)} + d_j^{(2)} + \cdots + d_j^{(r)},
\quad \text{ for all } j \in \{0, 1, \ldots, n_p\}.
\end{equation}

Now, suppose $p$ does not divide $b_k = b_k\left(e_1, \ldots, e_{mp}\right)$. We complete the proof of (b) by showing $a_k = a_k\left(e_1, \ldots, e_{mp}\right)$ is divisible by $p$.
Since the multinomial coefficient 
$b_k$ given above has numerator $(mp)!$, we deduce from Corollary~\ref{kummerbincor} with $v = 0$ that each $e_j$ is divisible by $p$.  Since $k < mp$, we have $e_{1} + 2 e_{2} + \cdots + mp e_{mp} = mp-k > 0$ so that at least one of $e_1, \ldots, e_{mp}$ is positive.  Suppose such an $e_j$ is $e_{j'}$.  Since $p \mid e_{j'}$, we deduce $e_{j'} \ge p$.  On the other hand, there are $e_{j'}$ occurrences of $j'!$ in the denominator of the multinomial coefficient $a_k$, and it is impossible to add a positive integer to itself $p$ times in base $p$ without having a carry (just consider what happens to the the right-most digit in base $p$ during the additions or refer to \eqref{digitsomeforpartb}).  We deduce $p \mid a_k$.  
\end{proof}

We are thus left with determining when (c) occurs in order to determine when $E_{mp}^{(mp)}(x)$ is in Eisenstein form with respect to $p$.  As (c) is a result about the constant term of $E_{mp}^{(mp)}(x)$, we study this constant term next.
To finish the proof of Theorem~\ref{eisensteinpolynsthm}, 
we want to show that $p^2 \mid E_{mp,0}$ if and only if $p$ divides $m (2^m-1) B_m$.


 \section{The constant term of $ E_{mp}^{(mp)}(x)$}

To show part (c), we make use of work of G.~D.~Liu and W.~P.~Zhang \cite{LiuZhang} on generalized Euler numbers, which we will write as $\overline{E}_{2n}^{(x)}$ using a slightly different notation than in \cite{LiuZhang} to avoid confusion with the generalized Euler polynomials. The authors in \cite{LiuZhang} define $\overline{E}_{2n}^{(x)}$ through the equation
\[
\bigg( \dfrac{2}{e^t + e^{-t}} \bigg)^{x} = \sum_{n=0}^{\infty} \overline{E}_{2n}^{(x)} \dfrac{t^{2n}}{(2n)!}.
\]
Observe that the left-hand side above is an even function, so its Maclaurin series only involves terms of even degree in $t$ as shown.  Also, by taking $t = 0$, one can see that $\overline{E}_{0}^{(k)} = 1$ for all positive integers $k$. 
Note that $\overline{E}_{2n}^{(1)}$ denotes the classical $(2n)^{\rm th}$ Euler number.

Define the Stirling numbers of the first kind $s(n,k)$ for integers $n$ and $k$ with $n \ge k \ge 0$ by the double recurrence relations
\begin{gather*}
s(n,0) = 0, \quad \forall n \ge 1, \qquad
s(n,n) = 1, \quad \forall n \ge 0,
\qquad \text{and}\\
s(n,k) = s(n-1,k-1) - (n-1) s(n-1,k), \quad
\forall n > k \ge 1.
\end{gather*}
Also, define the central factorial numbers $T(n,k)$ for integers $n$ and $k$ with $n \ge k \ge 0$ by the double recurrence relations
\begin{gather*}
T(n,0) = 0, \quad \forall n \ge 1, \qquad
T(n,n) = 1, \quad \forall n \ge 0,
\qquad \text{and}\\
T(n,k) = T(n-1,k-1) + k^2 T(n-1,k), \quad
\forall n > k \ge 1.
\end{gather*}
Though other definitions of these numbers would suffice for our purposes,  these recurrence relations help emphasize that the numbers $s(n,k)$ and $T(n,k)$ are integers.  Following \cite{LiuZhang}, for integers $n$ and $k$ with $n \ge k \ge 1$, we also define
\[
\rho(n,k)=(-1)^k\sum_{j=k}^n\dfrac{(2j)!}{2^j j!}s(j,k)T(n,j).
\]
As $2^j j!$ can be viewed as the product of the even positive integers $\le 2j$, we deduce that $\rho(n,k) \in \mathbb Z$.  The following is due to G.~D.~Liu and W.~P.~Zhang (see Theorem~2.1, the sentence after (2.18), and (3.16) in \cite{LiuZhang}).

\begin{theorem}[Liu and Zhang \cite{LiuZhang}, 2008]\label{liuthm1}
Let $n$ and $k$ be positive integers. 
Then 
\[
\overline{E}_{2n}^{(k)} = \sum_{i=1}^n\rho(n,i)\,k^i.
\]
Furthermore, 
\[
\rho(n,1) = -\overline{E}_{2n-2}^{(2)}
= -\dfrac{2^{2n-1} (2^{2n}-1) B_{2n}}{n}.
\]
\end{theorem}

Since $\rho(n,i) \in \mathbb Z$ for every $i \in \{ 1, 2, \ldots, n \}$, we deduce from the first equation in Theorem~\ref{liuthm1} that 
$\overline{E}_{2n}^{(k)} \in \mathbb Z$ for all positive integers $n$ and $k$.  Recall also that $\overline{E}_{0}^{(k)} = 1$ for every integer $k \ge 1$.  
We turn now to connecting the numbers $\overline{E}_{2n}^{(x)}$ to the constant term $E_{mp,0} = E_{mp}^{(mp)}(0)$ in our generalized Euler polynomial $E_{mp}^{(mp)}(x)$. 

As before, we take $m$ to be an even positive integer and $p$ to be an odd prime.  
By setting $x=0$ in \eqref{eulerdef} and replacing $t$ with $2t$, we see that
\begin{equation}\label{ourconstanttermeq}
\begin{aligned}
\bigg( \sum_{n=0}^{\infty} \overline{E}_{2n}^{(mp)} \dfrac{t^{2n}}{(2n)!} \bigg)
\bigg( \sum_{j=0}^{\infty} \dfrac{(-mp)^j t^j}{j!} \bigg) 
&= \left(\dfrac{2}{e^{t}+e^{-t}}\right)^{mp} e^{-mpt} \\[5pt]
&= \left(\dfrac{2}{e^{2t}+1}\right)^{mp} \\[5pt]
&= \sum_{n=0}^{\infty}E_n^{(mp)}(0)\dfrac{(2t)^n}{n!}.
\end{aligned}
\end{equation}
To obtain the term of degree $t^{mp}$ in the product on the left, we want to add terms of the form
\[
\overline{E}_{mp-j}^{(mp)} \dfrac{t^{mp-j}}{(mp-j)!} \cdot \dfrac{(-mp)^j t^j}{j!} = (-1)^{j} \binom{mp}{j} (mp)^j \,\overline{E}_{mp-j}^{(mp)} \dfrac{t^{mp}}{(mp)!},
\]
where $0 \le j \le mp$. 
Therefore, from \eqref{ourconstanttermeq}, we obtain
\[
2^{mp} E_{mp}^{(mp)}(0) = \sum_{j=0}^{mp} (-1)^{j} \binom{mp}{j} (mp)^j \,\overline{E}_{mp-j}^{(mp)}.
\]
Since $\overline{E}_{mp-j}^{(mp)} \in \mathbb Z$ for each $j$ in the sum, we obtain the congruence
\begin{equation}\label{prelimcongpartc}
2^{mp} E_{mp}^{(mp)}(0)
\equiv \overline{E}_{mp}^{(mp)} \pmod{p^2}.
\end{equation}

Since $m$ is even and $\rho(n,i) \in \mathbb Z$, from Theorem~\ref{liuthm1}, we see that
\begin{equation}\label{ebarmodpsquared}
\overline{E}_{mp}^{(mp)} 
=\sum_{i=1}^{mp/2}\rho\left(\dfrac{mp}{2},i\right)(mp)^i
\equiv mp \,\rho\left(\dfrac{mp}{2},1\right) \pmod{p^2},
\end{equation}
which in particular implies from \eqref{prelimcongpartc} that $E_{mp}^{(mp)}(0)$ is divisible by $p$. 
Furthermore, Theorem~\ref{liuthm1} implies
\begin{equation}\label{rhoequality}
\rho\left(\dfrac{mp}{2},1\right) 
= -\dfrac{2^{mp-1} (2^{mp}-1) B_{mp}}{mp/2}
= -\dfrac{2^{mp} (2^{mp}-1) B_{mp}}{mp}.
\end{equation}

We consider now two cases, depending on whether $p-1$ divides divides $m$ or not, beginning with the latter.

\vskip 5pt \noindent
\textbf{Case 1.}  \textit{$p-1$ does not divide $m$.}
\vskip 5pt
Since $p-1$ does not divide $m$, Kummer's congruence \cite{kummer0} (cf.~Corollary~2 in \cite{adelberg1}) implies
\[
\dfrac{B_{mp}}{mp} \equiv \dfrac{B_m}{m} \pmod{p}.
\]
Also, Fermat's Little Theorem gives us 
\[
2^{mp} \equiv 2^{m} \pmod{p} 
\qquad \text{and} \qquad
2^{mp}-1 \equiv 2^{m}-1 \pmod{p}.
\]
Combining the above with \eqref{prelimcongpartc}, \eqref{ebarmodpsquared} and \eqref{rhoequality}, we deduce
\[
E_{mp}^{(mp)}(0) \equiv -p \,(2^{m}-1) B_m \pmod{p^2}. 
\]

In this case, $E_{mp}^{(mp)}(x)$ is in Eisenstein form with respect to $p$ if and only if $p \nmid 2(2^{m}-1) B_m$, where the extra factor of $2$ is simply to ensure $2(2^{m}-1) B_m$ is an integer.  On the other hand, if $p \mid m$, then by an observation going back to J.~C.~Adams~\cite{adams} and which follows from Kummer's congruence noted above, we have that $p$ divides the numerator of $B_m$ so that $p \mid 2(2^{m}-1) B_m$.  Recalling $m$ is even, it follows that, for this case, $E_{mp}^{(mp)}(x)$ is in Eisenstein form with respect to $p$ if and only if $p \nmid m (2^{m}-1) B_m$.

\vskip 5pt \noindent
\textbf{Case 2.}  \textit{$p-1$ divides $m$.}
\vskip 5pt
The main difference in this case is that the von Staudt-Clausen Theorem (cf.~\cite{adelfil}) implies that $p$ exactly divides the denominator of $B_{m}$ and $B_{mp}$ (that is, $p$ divides these denominators and $p^2$ does not).  We will return to using this information shortly.  

If $p \mid m$, then \eqref{ebarmodpsquared} implies $\overline{E}_{mp}^{(mp)}$ is divisible by $p^2$ since $\rho(mp/2,1) \in \mathbb Z$.  From \eqref{prelimcongpartc}, we deduce $p^2 \mid E_{mp}^{(mp)}(0)$ so that $E_{mp}^{(mp)}(x)$ is not in Eisenstein form with respect to $p$. 

Suppose now $p \nmid m$.  Let $e$ be the order of $2$ modulo $p$, and let $r = \nu_{p}(2^e - 1)$.  Note then that 
\[
p^r \mid (2^e - 1)
\qquad \text{and} \qquad
p^{r+1} \nmid (2^e - 1),
\]
and $e$ is the order of $2$ modulo $p^r$.  
We claim that the order of $2$ modulo $p^{r+j}$ is $e p^{j}$ for every integer $j \ge 0$.  It suffices to show that 
\begin{equation}\label{case2indhyp}
p^{r+j} \mid \big( 2^{e p^{j}}  - 1 \big)
\qquad \text{and} \qquad
p^{r+j+1} \nmid \big( 2^{e p^{j}}  - 1 \big),
\end{equation}
for every $j \ge 0$.  The case $j = 0$ holds from the above. We give an induction argument, supposing now that for some integer $j_0 \ge 0$, we know \eqref{case2indhyp} holds with $j = j_0$.  Then there is an integer $s$ such that
\[
2^{e p^{j_0}} = 1 + s p^{r+j_0}
\qquad \text{and} \qquad
p \nmid s.
\]
We deduce, from the Binomial Theorem, that
\[
2^{ep^{j_0+1}} = (1+sp^{r+j_0})^p \equiv 1 + s p^{r+j_0+1} \pmod{p^{r+j_0+2}}.
\]
We obtain from this that \eqref{case2indhyp} holds with $j = j_0+1$, establishing \eqref{case2indhyp} for all $j \ge 0$ by induction.
Thus, the order of $2$ modulo $p^{r+j}$ is $e p^{j}$ for every integer $j \ge 0$.

Since $e$ divides $p-1$ and $p-1$ divides $m$ but $p \nmid m$, we deduce that 
\[
p^r \mid (2^m-1), \quad
p^{r+1} \nmid (2^m-1), \quad
p^{r+1} \mid (2^{mp}-1), \quad \text{and} \quad
p^{r+2} \nmid (2^{mp}-1).
\]
We consider two possibilities depending on whether $r = 1$ or $r > 1$.

For the first possibility, where $r = 1$, in \eqref{rhoequality}, we have
$B_{mp}$ has a denominator exactly divisible by $p$, 
$p$ also exactly divides the denominator $mp$ since $p \nmid m$, and the expression $2^{mp}(2^{mp}-1)$ is exactly divisible by $p^{r+1} = p^{2}$.  Thus, \eqref{rhoequality} implies that $\rho(mp/2,1)$ is not divisible by $p$,
and \eqref{prelimcongpartc} and \eqref{ebarmodpsquared} in turn imply that $E_{mp}^{(mp)}(x)$ is in Eisenstein form with respect to $p$.  Furthermore, $p \nmid m$, $p^r = p$ exactly divides $2^m-1$, and $p$ exactly divides the denominator of $B_m$ so that $p$ does not divide $m (2^m-1) B_m$.  

Now, consider the possibility that $r > 1$.  Then $B_{mp}$ has a denominator exactly divisible by $p$ and the expression $2^{mp}(2^{mp}-1)$ is divisible by $p^{r+1}$ and hence $p^3$.  We deduce from \eqref{rhoequality} that $\rho(mp/2,1)$ is divisible by $p$, so that \eqref{prelimcongpartc} and \eqref{ebarmodpsquared} imply that $E_{mp}^{(mp)}(x)$ is not in Eisenstein form with respect to $p$.  Here, $2^m-1$ is divisible by $p^2$
and $p$ exactly divides the denominator of $B_m$ so that $p$ divides $m (2^m-1) B_m$. 

Combining the above, we see that in the case $p-1$ divides $m$, we have that $E_{mp}^{(mp)}(x)$ is in Eisenstein form with respect to $p$ if and only if $p \nmid m (2^m-1) B_m$.


\section{Proof of Theorem~\ref{positivepropthm}}

In this section, we justify Theorem~\ref{positivepropthm}.  We begin with the following, which is contained in the argument for Lemma~2 in \cite{adelfil}.  

\begin{lemma}\label{thm2lemma1}
The inequality $|2(2^{m}-1) B_{m}| \le m^{m}$ holds for every $m \in \mathbb Z^{+}$.
\end{lemma}

\begin{proof}
We use that $B_{1} = -1/2$, and for $k \ge 1$, we have $B_{2k+1} = 0$  and 
\[
B_{2k} = (-1)^{k-1} \dfrac{2 (2k)!}{(2\pi)^{2k}} \zeta(2k)
\]
(cf.~\cite{borshaf}).  For $m = 1$, we have $|2(2^{m}-1) B_{m}| = 1 = m^{m}$, so the stated inequality holds.  
For $m = 2k \ge 2$, we have $|\zeta(2k)| \le |\zeta(2)| < 2$ so that 
\[
|2(2^{m}-1) B_{m}| \le 2(2^{m}-1)  \dfrac{4 (m!)}{(2\pi)^{m}} < \dfrac{8(m!)}{\pi^{m}} < m^{m},
\]
completing the proof.
\end{proof}

The basic idea is to consider the positive integers $n = mp \le t$ where $m \in \mathbb Z^{+}$ and $p > m$ is a prime.  
For such $n$, Theorem~\ref{eisensteinpolynsthm} implies that if $p$ does not divide the integer $2 (2^{m}-1) B_{m}$, 
then $E_{mp}^{(mp)}(x)$ is in Eisenstein form with respect to $p$.  We show that most such $n$ (asymptotically almost all) are such that $E_{n}^{(n)}(x)$ is in Eisenstein form with respect to the largest prime divisor $p$ of $n$.  

Fix $\varepsilon \in (0,1/2)$. 
Let $\mathcal S = \mathcal S(\varepsilon,t)$ be the set of pairs $(m,p)$ with 
\begin{gather*}
m \in \mathbb Z^{+}, \quad m \le t^{(1/2)-\varepsilon}, \quad m \text{ even}, \\
p > m \text{ a prime}, \quad mp \le t, \quad \text{ and } \quad p \nmid 2 (2^{m}-1) B_{m}.
\end{gather*}
By Theorem~\ref{eisensteinpolynsthm}, the pairs $(m,p) \in \mathcal S$ correspond to unique positive integers $n = mp \le t$ for which $E_{n}^{(n)}(x)$ is in Eisenstein form.  
Our interest is in counting the number of pairs in $\mathcal S$.
 
We make use of the notation $\pi(x)$ for the number of primes $\le x$ and $f(t) = (1 + o(1)) g(t)$ to indicate that
for every fixed $\varepsilon' > 0$ and $t \ge t_{0}(\varepsilon')$ sufficiently large, we have $(1-\varepsilon') g(t) < |f(t)| < (1+\varepsilon') g(t)$.  
For a fixed $m \in \mathbb Z^{+}$ with $m \le t^{(1/2)-\varepsilon}$ and $m$ even, 
there are $\pi(t/m) - \pi(m)$ different primes $p > m$ for which $mp \le t$.  
By Lemma~\ref{thm2lemma1}, there are $< m$ primes $p > m$ which divide $2(2^{m}-1) B_{m}$.  
From the Prime Number Theorem, the number of pairs $(m,p) \in \mathcal S$, with $m$ still fixed, is at least
\begin{align*}
\pi\bigg(  \dfrac{t}{m}  \bigg) - \pi(m) - m 
&\ge (1+o(1)) \dfrac{t}{m \log (t/m)} - 2 t^{(1/2)-\varepsilon} \\[5pt] 
&= (1+o(1)) \dfrac{t}{m \log (t/m)},
\end{align*}
where the equality follows from
\[
\dfrac{t}{m \log (t/m)} 
\ge \dfrac{t}{t^{(1/2)-\varepsilon} \log t}
= \dfrac{t^{(1/2)+\varepsilon}}{\log t}.
\]
The above only depends on $t/m \ge t^{(1/2)+\varepsilon}$ being sufficiently large compared to $t^{(1/2)-\varepsilon}$, and in particular the $o(1)$ notation is uniform in $m \le t^{(1/2)-\varepsilon}$.  
We obtain that 
\[
|\mathcal S| \ge (1+o(1)) \sum_{\substack{m \le t^{(1/2)-\varepsilon} \\ m \text{ even}}} \dfrac{t}{m \log(t/m)}
=  (1+o(1)) \log\bigg( \dfrac{2}{1+2\varepsilon} \bigg) \dfrac{t}{2},
\]
where the latter can be deduced from a comparison of the sum to the integral
\[
\int_{1}^{z} \dfrac{t}{(2x) \log(t/(2x))} \,dx
= -\dfrac{t \log\log(t/(2x))}{2} \bigg|_{1}^{z}
\]
where $z = (1/2) \,t^{(1/2)-\varepsilon}$.  
As the above holds for each $\varepsilon \in (0,1/2)$, Theorem~\ref{positivepropthm} follows.

\vskip 8pt \noindent
\textbf{Acknowledgment:}  The authors thank Pieter Moree for pointing out connections of our work to \cite{moree1,moree2}, leading to Theorem~\ref{newthm}.


\end{document}